\documentclass[12pt]{amsart}
\usepackage{a4wide}
\usepackage{amssymb}
\usepackage{amsthm} 
\usepackage{amsmath} 
\usepackage{amscd}
\usepackage{verbatim}
\usepackage{mathrsfs}
\usepackage{dsfont,bbm} 
\usepackage{youngtab}
\usepackage[usenames]{color}
\usepackage{tikz,graphicx}
\usepackage{hyperref}
\usepackage{multicol}

\numberwithin{equation}{section}
\theoremstyle{plain}
\newtheorem{theorem}{Theorem}[section]

\newtheorem{lemma}[theorem]{Lemma}

\newtheorem*{questionA}{Question A}
\newtheorem*{questionB}{Question B}
\newtheorem*{questionC}{Question C}
\theoremstyle{definition}

\newtheorem*{example}{Example}
\newtheorem*{conjecture}{Conjecture}

\theoremstyle{remark}
\newtheorem*{remark}{Remark}

\DeclareMathOperator{\SL}{SL}

\newcommand{\U}{\mathbb{U}}

\newcommand{\Z}{\mathbb Z}

\newcommand{\F}{\mathbb F}

\newcommand{\M}{\mathbb{M}}

\newcommand{\Ogg}{\rm {Ogg}}

\newcommand{\Vnat}{V^{\natural}}

\newcommand{\MM}{\mathbb{M}}

\newcommand{\RR}{\mathbb R}

\renewcommand{\t}{\tau}
\usepackage{ctable}
\newcommand{\tr}{{\rm tr}}

\begin{document}

\title[The Jack Daniels Problem]
{The Jack Daniels Problem}

\author{John F. R. Duncan and Ken Ono}

\address{Department of Mathematics, Applied Mathematics and Statistics,
Case Western Reserve University, Cleveland, Ohio 44106}
\email{john.duncan@case.edu}

\address{Department of Mathematics and Computer Science,
Emory University, Atlanta, Georgia 30322}
\email{ono@mathcs.emory.edu}

\thanks{The authors are supported by the NSF. The second  author  thanks the support of the Asa Griggs Candler Fund.}

\subjclass[2010]{11F11, 11G05, 20C34}

\dedicatory{For Winnie Li and her advisor Andy Ogg}
\begin{abstract} 
In 1975 Ogg offered a bottle of Jack Daniels for an explanation of the fact that
the prime divisors of the order of the monster $\M$ are the primes $p$
for which the characteristic $p$ supersingular $j$-invariants are all defined over $\F_p$.
This coincidence is often suggested as the first hint of {\it monstrous moonshine}, the deep unexpected interplay between the monster and
modular functions. We revisit Ogg's problem, and we point out (using existing tools)
that the moonshine functions for order $p$
elements give the
set of characteristic $p$
supersingular $j$-invariants (apart from 0 and 1728). 
Furthermore, we discuss this coincidence
of the two seemingly unrelated sets of primes using the first principles of moonshine.

\end{abstract}
\maketitle

\section{Introduction}

At a seminar at the Coll\`ege de France in 1975, Tits gave the order of  the {\it monster}\footnote{The monster was only conjectured to exist at the time. In a tour de force, Griess famously constructed the monster in the early 1980s
\cite{Griess1, Griess2}.} group $\M$, the largest sporadic finite
simple group. It is the integer
$$
\#\M= 2^{46}\cdot 3^{20}\cdot 5^9\cdot 7^6\cdot 11^2\cdot 13^3\cdot 17\cdot 19\cdot 23\cdot 29
\cdot 31\cdot 41\cdot 47\cdot 59\cdot 71.
$$
Ogg noticed \cite{Ogg} that the prime divisors are 
 primes $p$  for which the characteristic $p$ supersingular $j$-invariants are all defined over $\F_p$, the set we denote by
\begin{equation}\label{ssprimes}
\Ogg_{ss}:=\left \{ 2, 3, 5, 7, 11, 13, 17, 19, 23, 29, 31, 41, 47, 59, 71\right\}.
\end{equation}
Ogg offered a bottle of Jack Daniels\footnote{{\it ``Une bouteille de Jack Daniels est offerte \`a celui qui expliquera cette coincidence.}'' (p. 7 of  \cite{Ogg}).} for an explanation of this coincidence.
Although this problem has not been addressed in the literature (to our knowledge), most experts agree that the explanation is the proof of monstrous moonshine by Borcherds.\footnote{Borcherds has not claimed the bottle of Jack Daniels.}

Here we revisit Ogg's question.
Loosely speaking, we show that the monster module knows the supersingular $j$-invariants in characteristic $p$ for precisely
the primes $p\in \Ogg_{ss}$. 
We have reformulated Ogg's problem as three questions.

\begin{questionA} 
Is there a natural method of producing the characteristic $p$ supersingular $j$-invariants (other than $0$ and $1728$) from elements of order $p$ in $\M$?
\end{questionA}

We shall show that there is indeed such a rule for the primes $p\mid \# \M$. 
This rule arises from {\it monstrous moonshine}, the deep unexpected
interplay between the monster and modular functions. 
Ogg's observation is often suggested as the first hint of moonshine,
and so it makes sense to revisit Ogg's question
from the perspective of the first principles of moonshine.
We offer answers to the following two questions.

\begin{questionB}
If $p\not \in \Ogg_{ss}$ is prime, then why would one expect a priori that $p\nmid \# \M$?
\end{questionB}

\begin{questionC}
If $p\in \Ogg_{ss}$ is prime, then why would one expect a priori that $p\mid \#\M$?
\end{questionC}

We begin by recalling aspects of the monstrous moonshine
conjecture, which was formulated \cite{MR554399} by Conway and Norton in 1978, and proven
\cite{borcherds_monstrous} by Borcherds in 1992. 
(We refer to the introduction of \cite{FLM}, and the more recent surveys \cite{DuncanGriffinOno, Gannon1} for more detailed accounts.)
As we will explain, the monstrous moonshine conjecture states that for each $g\in \M$ there is an associated McKay--Thompson series $T_g(\tau)$, which
is a  distinguished modular function, and also a 
conjugacy class invariant.
We will answer Question~A in the affirmative using McKay--Thompson series for elements of prime 
order
in $\M$. 
The solution is simple and uniform.
It turns out that monstrous moonshine provides the definitive answer to Question~B, one which requires
Ogg's own work. 


Let $j(\tau)$ denote the usual modular $j$-function
\begin{equation}
j(\tau)-744=\sum_{n=-1}^{\infty}c(n)q^n=
q^{-1}+196884q+21493760q^2+ 864299970q^3+\dots.
\end{equation}
(Note $q:=e^{2\pi i \tau}$.)
Monstrous moonshine has its origins in the famous observations 
\begin{displaymath}
\begin{split}
	196884&=1+196883,\\
	21493760&=1+196883+21296876,\\
	864299970&=1+1+196883+196883+21296876+842609326.
	\end{split}
\end{displaymath}
The right hand sides of these expressions are sums of dimensions of irreducible representations of the monster $\MM$. The irreducible representations appearing are just the first four (when ordered by size), of a total of $194$ in the character table of $\MM$. Namely,
\begin{displaymath}
1,\ \ 
196883, \ \ 
21296876,\ \ 
842609326,\ \ 
\dots \dots,\ \
258823477531055064045234375.
\end{displaymath}

Based on these observations, Thompson conjectured that there is a naturally defined
graded infinite-dimensional monster module, denoted 
\begin{equation}
\Vnat=\bigoplus_{n=-1}^{\infty}\Vnat_n,
\end{equation}
which satisfies
$\dim(\Vnat_n)=c(n)$
for $n\geq -1$. Later, $\Vnat$ was constructed explicitly by 
Frenkel, Lepowsky and Meurman \cite{FLMPNAS, FLMBerk, FLM}.
Thompson also suggested to consider the graded-trace functions
\begin{gather}\label{eqn:intro-Tg}
	T_g(\tau):=\sum_{n=-1}^{\infty}\tr(g|\Vnat_n)q^n,
\end{gather}
for $g\in \MM$, the so-called {\em McKay--Thompson series}.
Conway and Norton followed his suggestion, and they  \cite{MR554399}
 formulated the {\em monstrous moonshine conjecture}:

\begin{conjecture}[Conway--Norton]\label{conj:MonstrousMoonshine}
For each $g\in\MM$ there is a specific
group $\Gamma_g<\SL_2(\RR)$ such that $T_g(\tau)$ is the normalized Hauptmodul for $\Gamma_g$.
\end{conjecture}

\medskip
\noindent
{\it Four Remarks.}

\smallskip

\noindent
1) If $e\in \M$ is the identity, then we have $T_e(\tau)=j(\tau)-744$.

\smallskip
\noindent
2) 
In the course of formulating their conjecture, Conway and Norton introduced the notion of replicability of
modular functions (see Section~8 of \cite{MR554399}).
Their conjecture requires that for each prime
 $p\mid \#M$ there is a corresponding order $p$ element, say $g_p\in \M$, for which
$\Gamma_{g_p}=\Gamma_0(p)^+$, where $\Gamma_0(p)^+$ is the extension of
$\Gamma_0(p)$ by $w_p:=\frac{1}{\sqrt{p}}\left(\begin{smallmatrix} 0&-1\\ p&0\end{smallmatrix}\right)$.

\smallskip
\noindent
3) Note that most of the $T_g$ were given non-zero constant terms in \cite{MR554399}. See Section 6 of \cite{MR760657} for an explanation of how these were chosen. It is now understood that the $T_g$ should be {\em normalized} Hauptmoduls, satisfying $T_g(\tau)=q^{-1}+O(q)$ as $\Im(\tau)\to \infty$ for every $g\in \MM$. 

\smallskip
\noindent
4) The function $T_{g_p}(\t)$ for $g_p\in \MM$ such that $\Gamma_{g_p}=\Gamma_0(p)^+$ is denoted $T_{p+}(\t)$ in \cite{MR554399}.

\medskip


Borcherds famously proved the monstrous moonshine conjecture in \cite{borcherds_monstrous}, by considering the denominator identity of the monster Lie algebra, which is a Borcherds--Kac--Moody algebra he constructed using the vertex operator algebra structure on the moonshine module $\Vnat$ of Frenkel, Lepowsky and Meurman.

Monstrous moonshine answers Question~B immediately. Ogg completely classified
the genus zero modular curves $X_0(p)^+$, and he proved \cite{Ogg} that they are the ones for which
$p\in \Ogg_{ss}$. Therefore, if
$p\not \in \Ogg_{ss}$ is prime, then $X_0(p)^+$
has positive genus, which in turn means that there is no corresponding Hauptmodul.
Therefore, if $p\not \in \Ogg_{ss}$, then Monstrous Moonshine implies that $p\nmid \# \M$.

We turn to Question~A. 
For each prime $p\mid \# \M$ we choose an order $p$ element  $g_p\in \M$,
where $\Gamma_{g_p}=\Gamma_0(p)^+$, and we study the corresponding McKay--Thompson series. We prove that
these McKay--Thompson series know the characteristic $p$ supersingular $j$-invariants which
differ from $0$ and $1728$. \medskip
For the primes $p\mid \# \M$, we let
\begin{equation}
\U_{g_p}(\tau):= T_{g_p}(\tau) \ | \ U(p),
\end{equation}
where
$$
\left (\sum_{n\gg -\infty} a(n)q^n\right) \ | \ U(p):=\sum_{n\gg -\infty} a(pn)q^n.
$$
We have the following theorem which gives the $\U_{g_p}(\tau)\pmod p$ as reciprocals of monic linear functions of $j(\tau)$.
This is our solution to Question~A.

\begin{theorem}\label{mainthm}
Suppose that $p\mid \# \M$ is prime. Then the following are true:

\begin{enumerate}
\item  We have that $\U_{g_p}(\tau)\pmod p$  is a weight $p-1$ cusp form modulo $p$ on $\SL_2(\Z)$.

\smallskip
\item  We have that $\U_{g_p}(\tau)\pmod p$ is a $\Z$-linear sum of reciprocals of monic linear polynomials of $j(\tau)$.
The roots of these polynomials constitute the set of supersingular $j$-invariants in characteristic $p$ which do not equal $0$ and $1728$.
\end{enumerate}
\end{theorem}

\medskip
\noindent
{\it Two remarks.}

\smallskip

\noindent
1) We give the exact expressions of $\U_{g_p}(\tau)\pmod p$ in two tables in the Appendix.
One table gives expressions in terms of reciprocals of monic polynomials in $j(\tau)$, and the
other table gives expressions in terms of weight $p-1$ cusp forms on $\SL_2(\Z)$.

\smallskip
\noindent
2)  There are no supersingular $j$-invariants, apart from possibly $0$ and $1728$, for the primes $p\leq 11$.
This explains the vanishing of $\U_{g_p}(\tau)\pmod p$ in these cases.

\bigskip

\begin{example} For the prime $p=71$, we have 
$$
T_{g_{71}}(\tau)=\frac{\Theta(4,2,18;\tau)-\Theta(6,2,12;\tau)}{2\eta(\tau)\eta(71\tau)}=
q^{-1}+q+q^2+q^3+q^4+2q^5+2q^6+3q^7+\dots.
$$
Here $\eta(\tau)$ is the usual Dedekind eta-function, and
$$\Theta(a,b,c;\tau):=\sum_{x,y\in \Z}q^{\frac{1}{2}\left(ax^2+bxy+cy^2\right)}.
$$
By direct calculation, we find that
\begin{displaymath}
\begin{split}
&\U_{g_{71}}(\tau)=2773q+302729q^2+12173239q^3+285152905q^4+4692994938q^5+\dots\\
&\ \ \ \ \equiv 4q+56q^2+5q^3+7q^4+18q^5+67q^6+66q^7+55q^8+47q^9+68q^{10}+\dots \pmod{71}.
\end{split}
\end{displaymath}
On the other hand, using the standard Eisenstein series $E_4(\tau)$ and $E_6(\tau)$, and the weight 12 cusp form
$\Delta(\tau)$, we have
\begin{displaymath}
\begin{split}
&4E_4E_6^9\Delta+12E_4E_6^7\Delta^2+65E_4E_6^5\Delta^3+11E_4E_6^3\Delta^4+35E_4E_6\Delta^5\\
&\ \ \ \ \equiv \frac{18}{j(\tau)+5}+\frac{48}{j(\tau)+23}+\frac{16}{j(\tau)+30}+\frac{40}{j(\tau)+31}+ \frac{24}{j(\tau)+54}\\
&\ \ \ \ \equiv 4q+56q^2+5q^3+7q^4+18q^5+67q^6+66q^7+55q^8+47q^9+68q^{10}\dots \pmod{71}.
\end{split}
\end{displaymath}
The supersingular $j$-invariants, apart from $0$ and $1728$, are
$$-5, -23, -30, -31, -54\pmod{71}.
$$
They are all defined over $\F_{71}$.
\end{example}

\smallskip

In Section~\ref{proof} we easily derive Theorem~\ref{mainthm}
from monstrous moonshine, the replicability of McKay--Thompson series, and a classical
result of Dwork and Koike on the $p$-adic rigidity of the $j$-function.
In Section~\ref{C} we shall discuss our answer to Question~C.

\section*{Acknowledgements}
\noindent
The authors thank Igor Frenkel, Bob Griess and Andy Ogg for their comments.

\section{Proof of Theorem~\ref{mainthm}}\label{proof}
Theorem~\ref{mainthm}  is a simple observation, obtained by 
combining replicability identities in montrous moonshine
with a classical result of Dwork and Koike on the $p$-adic rigidity of $j(\tau)$.
The notion of replicability (see Section~8 of \cite{MR554399}) explains
aspects of the the group law in $\M$ (namely, the power maps, $g\mapsto g^p$) in terms of Fourier expansions of
McKay--Thompson series.  Here we give one instance
of these identities.

\begin{lemma}\label{replicability}
Suppose that $g\in \M$, and that $p$ is a prime for which 
$w_p=\frac{1}{\sqrt{p}}\left(\begin{smallmatrix} 0&-1\\p & 0 \end{smallmatrix}\right)\in \Gamma_g$.
Then the following are true:

\begin{enumerate}
\item We have that
$$
T_g(\tau)+ p \left(T_g(\tau) \ | \ U(p)\right) = T_{g^p}(\tau).
$$

\item
If moreover $g^p=e$, then 
$$
T_g(\tau) + p \left(T_g(\tau) \ | \ U(p)\right)= j(\tau)-744.
$$
\end{enumerate}
\end{lemma}
\begin{proof}
Under these hypotheses, the McKay--Thompson series $T_g(\tau)$ and $T_{g^p}(\tau)$ are related by the {\it replicable} identity (see p. 318 of \cite{MR554399})
\begin{displaymath}
T_g(\tau)+T_g\left(\frac{\tau}{p}\right)+T_g\left(\frac{\tau+1}{p}\right)+\dots+
T_g\left(\frac{\tau+p-1}{p}\right)=T_{g^p}(\tau).
\end{displaymath}
This immediately implies the claim that
\begin{displaymath}
T_g+ p T_g \ | \ U(p) = T_{g^p}.
\end{displaymath}
Moreover, if $g$ has order $p$, then (2) follows from the fact that
$T_e(\tau)=j(\tau)-744$.
\end{proof}

Now we recall a classical result of Dwork and Koike \cite{Dwork, Koike}, as reformulated
by Swisher \cite{Swisher}.
Suppose that
 $p\geq 5$ is prime, let $SS_p$ be the set of characteristic $p$ supersingular $j$-invariants in $\F_p\setminus \{0, 1728\}$, and
 let $SS^*_p$ be the set of monic irreducible quadratic polynomials $g(x)\in F_p[x]$ whose roots are
 the supersingular $j$-invariants (if any) in $\F_{p^2}\setminus \F_p$.
 Then we have the following congruence for $\left (j(\tau)-744\right)\ | \ U(p)\pmod p$.
 
 \begin{lemma}\label{rigidity}
If $p\geq 5$ is prime, then
 for every $\alpha \in SS_p$  (resp. every $g(x)\in SS_p^*$) there is an integer $A_p(\alpha)$
(resp. pair of integers integers $B_p(g)$ and $C_p(g)$) for which
\begin{displaymath}
 \left(j(\tau)-744\right)\ | \ U(p) \equiv -\sum_{\alpha \in SS_p}\frac{A_p(\alpha)}{j(\tau)-\alpha}-
\sum_{g(x)\in SS^*_p} 
\frac{B_p(g)j(\tau)+C_p(g)}
{g(j(\tau))}
\pmod{p}.
\end{displaymath}
\end{lemma}
\begin{proof}
We begin by recalling Swisher's reformulation \cite{Swisher} of the result of Dwork and Koike. 
The content of Theorem 1.1 in \cite{Swisher} is that
\begin{displaymath}
\begin{split}
j(p\tau)\equiv p\left( j(\tau)-744\right) &\ | \ T(p)+744\\
&\ \ \ \ \ +p\sum_{\alpha \in SS_p}\frac{A_p(\alpha)}{j(\tau)-\alpha}+p
\sum_{g(x)\in SS^*_p} 
\frac{B_p(g)j(\tau)+C_p(g)}
{g(j(\tau))}
\pmod{p^2}.
\end{split}
\end{displaymath}
Here $T(p)$ is the usual $p$th Hecke operator of weight $0$\footnote{Swisher used a different normalization for her Hecke operators.} 
Since $$p \left( j(\tau)-744\right)\ | \ T(p)= j(p\tau)-744 +p\left( j(\tau)-744\ | \ U(p)\right),$$
and $T_e(\tau)=j(\tau)-744$, 
 we find that
\begin{displaymath}
\left( j(\tau)-744)\right) \ | \ U(p)\equiv -\sum_{\alpha \in SS_p}\frac{A_p(\alpha)}{j(\tau)-\alpha}-
\sum_{g(x)\in SS^*_p} 
\frac{B_p(g)j(\tau)+C_p(g)}
{g(j(\tau))}
\pmod{p}.
\end{displaymath}
\end{proof}

\begin{proof}[Proof of Theorem~\ref{mainthm}]
For each prime $p$ dividing $\#\M$ we choose an order $p$ element $g_p\in \M$ so that $\Gamma_{g_p}=\Gamma_0(p)^+$. In particular, $w_p\in \Gamma_{g_p}$.  Lemma~\ref{replicability} then implies that 
$$
T_{g_p}(\tau)\equiv j(\tau)-744\pmod p.
$$
If $p=2$ or $p=3$ then $(j(\t)-744))\mid U(p)\equiv 0\pmod p$, so $\U_{g_p}(\tau)\equiv 0\pmod p$ and the claims are true in these cases. 

By Ogg's theorem about $\Ogg_{ss}$, if $p\geq 5$, then $SS_p^*$ is empty since $p\in \Ogg_{ss}$.
Therefore,
 Lemma~\ref{rigidity} then implies that
\begin{displaymath}
\U_{g_p}(\tau):=\left( T_{g_p}(\tau)\ | \ U(p)\right)\equiv \left( j(\tau)-744\right)\ | \ U(p)
\equiv  -\sum_{\alpha \in SS_p}\frac{A_p(\alpha)}{j(\tau)-\alpha}\pmod{p}.
\end{displaymath}
This is claim (2).

It is well known that the divisor of the Eisenstein series $E_{p-1}(\tau)$ is the Hasse invariant
for the locus of supersingular $j$-invariants (for example, see \cite{Dwork}).
Moreover, since $E_{p-1}(\tau)\equiv 1\pmod p$, it follows that 
$$
E_{p-1}(\tau)\U_{g_p}(\tau)\equiv \U_{g_p}(\tau)\pmod p
$$
is the reduction modulo $p$ of a weight $p-1$ modular form on $\Gamma_0(p)$ which vanishes at the cusp $i \infty$ because
$\U_{g_{p}}(\tau)$ has vanishing constant term. Although $U(p)$ is not an operator on
level 1 forms, this form modulo $p$ is actually even on $\SL_2(\Z)$ because
$$
(j(\tau)-744) \ | \ U(p) \equiv p(j(\tau)-744) \ | \ T(p)\pmod p.
$$
Therefore, it is congruent to a weight $p-1$ cusp form on $\SL_2(\Z)$.
This is claim (1).
\end{proof}

\section{Discussion of Question C}\label{C}
The seminal paper of Conway and Norton
\cite{MR554399} contains 
tables of modular functions which are constructed from
elementary theta series. As a guiding principle, we suggest that such Hauptmoduls
are the ones expected to be McKay--Thompson series.

Suppose that $p\in \Ogg_{ss}$. Then $X_0(p)^+$ has genus zero, and there
is a normalized Hauptmodul $h_p(\tau)$. By the theory of Hecke operators, it follows that
 $h_p(\tau)+p\left (h_p(\tau)\ | \ U(p)\right) =q^{-1}+O(q)$ is a modular function on $\SL_2(\Z)$ which is holomorphic on the
 upper half of the complex plane, which means that
 $$
 h_p(\tau)+p\left (h_p(\tau)\ | \ U(p)\right) =j(\tau)-744.
 $$
 Therefore, we have
 \begin{equation}\label{eqnC}
 h_p(\tau)\ | \ U(p)\equiv \left( j(\tau)-744\right) \ | \ U(p)\pmod p.
 \end{equation}
  Arguing as in the proof of Theorem~\ref{mainthm} (1),
the right hand side of (\ref{eqnC}) is the mod $p$ reduction of a weight $p-1$ cusp form
on $\SL_2(\Z)$. 
By multiplying by $j'(\tau)=\frac{E_4^2(\tau)E_6(\tau)}{\Delta(\tau)}$,  we find that
\begin{equation}\label{eqnC2}
j'(\tau)\cdot \left( h_p(\tau)\ | \ U(p)\right)\pmod p
\end{equation}
is the reduction mod $p$ of a level 1 holomorphic modular form of weight $p+1$.

Deep work of Dong and Mason \cite{DongMason} proves that the monster module $V^{\natural}$
 produces many more modular objects than 
 Hauptmoduls.  One of their results is that all holomorphic modular forms of level 1
arise from vectors in $V^{\natural}$ (see Theorem~1 of \cite{DongMason}). 
In particular, vectors in $V^{\natural}$  yield the holomorphic modular forms
whose reductions mod $p$  appear in (\ref{eqnC2}).

We now recall the famous fact that the reduction mod $p$ of $S_2(\Gamma_0(p))$, the space of weight 2 cusp forms on
$\Gamma_0(p)$, is the reduction mod $p$ of $S_{p+1}$, the space of weight $p+1$ level 1 cusp forms. (For example, see
\cite{Serre}.)
Therefore, the level 1 weight $p+1$ forms whose reductions mod $p$ appear in (\ref{eqnC2})
can be described in terms of a basis of forms in $S_2(\Gamma_0(p))$ reduced mod $p$.

We expect that  $h_p(\tau)=T_g(\tau)$ for some McKay--Thompson series. If this is confirmed, then
Lemma~\ref{replicability}~(2) 
implies that $p\mid \#\M$. This is replicability.
According to the guiding principle, we  ask whether $h_p(\tau)$ can be described
by elementary theta functions.
It turns out that Pizer \cite{Pizer} already studied this question
for $S_2(\Gamma_0(p))$, which contains a form which satisfies the mod $p$ congruence above in (\ref{eqnC2}).
In 1978 he proved that the primes in  $\Ogg_{ss}$ are precisely the primes
for which $S_2(\Gamma_0(p))$ is spanned by Hecke's discriminant $p$ theta functions.
Therefore, we expect $h_p(\tau)$ to be similarly described globally. The guiding
principle then suggests that it is a McKay--Thompson series, which in turn would imply that
$p\mid \# \M$. 

\begin{remark} This discussion does not prove
that every $p\in \Ogg_{ss}$ divides $\# \M$. It merely explains how the first principles of
moonshine suggest this implication. 
Monstrous moonshine is the proof.
Does this then provide a completely satisfactory solution to Ogg's problem?
Maybe or maybe not. Perhaps someone will one day furnish a map from the characteristic $p$ supersingular $j$-invariants to elements of order $p$ where
the group structure of $\M$ is apparent. 
\end{remark}


\section*{Appendix}

Table 1 gives the names $pZ$ of the conjugacy classes $[g_p]\subset \MM$ such that $\Gamma_{g_p}=\Gamma_0(p)^+$. 
We follow the convention of writing $pAB$ for $pA\cup pB$.
Tables 2 and 3 express $\U_{g_p}(\tau)\pmod p$
in terms of supersingular $j$-invariants and level 1 cusp forms.

\smallskip
\begin{table}[h]
\caption{Conjugacy Classes in $\MM$}
\begin{tabular}{|r|c|}\hline
$p$& $[g_p]\subset\MM$\\ 
\hline
$2$& $2A$\\ 
$3$& $3A$\\ 
$5$& $5A$\\ 
$7$& $7A$\\ 
$11$& $11A$\\ 
$13$& $13A$\\ 
$17$& $17A$ \\ 
$19$& $19A$\\ 
$23$& $23AB$\\ 
$29$& $29A$\\ 
$31$& $31AB$\\ 
$41$& $41A$\\ 
$47$& $47AB$ \\ 
$59$& $59AB$\\ 
$71$& $71AB$\\ \hline
\end{tabular}
\end{table}

\begin{table}[ht]
\caption{$\U_{g_p}(\tau)\pmod{p}$}
\begin{tabular}{|r|l|}\hline
$p$& $\U_{g_p}(\tau) \pmod p$\\ 
\hline
$2$&$0$ \\ 
$3$&$0$ \\ 
$5$&$0$ \\ 
$7$&$0$ \\ 
$11$&$0$ \\ 
$13$&$\frac{12}{j(\tau)+8}$ \\ 
$17$&$\frac{4}{j(\tau)+9}$ \\ 
$19$&$\frac{7}{j(\tau)+12}$ \\ 
$23$&$\frac{4}{j(\tau)+4}$ \\ 
$29$&$\frac{9}{j(\tau)+4}+\frac{23}{j(\tau)+27}$ \\ 
$31$&$\frac{20}{j(\tau)+27}+\frac{7}{j(\tau)+29}$ \\ 
$41$&$\frac{36}{j(\tau)+9}+\frac{20}{j(\tau)+13}+\frac{31}{j(\tau)+38}$ \\ 
$47$&$\frac{32}{j(\tau)+3}+\frac{4}{j(\tau)+37}+\frac{17}{j(\tau)+38}$ \\ 
$59$&$\frac{21}{j(\tau)+11}+\frac{5}{j(\tau)+12}+\frac{4}{j(\tau)+31}+\frac{3}{j(\tau)+44}$ \\ 
$71$&$\frac{18}{j(\tau)+5}+\frac{48}{j(\tau)+23}+\frac{16}{j(\tau)+30}+\frac{40}{j(\tau)+31}+ \frac{24}{j(\tau)+54}$\\ \hline
\end{tabular}
\end{table}

\vskip1in

\begin{table}[ht]
\caption{$\U_{g_p}(\tau)\pmod{p}\in S_{p-1}\pmod p$}\label{Japprox}
\begin{tabular}{|r|l|}\hline
$p$& $\U_{g_p}(\tau) \pmod p$\\ 
\hline
$2$&$0$ \\ 
$3$&$0$ \\ 
$5$&$0$ \\ 
$7$&$0$ \\ 
$11$&$0$ \\ 
$13$&$E_4\Delta$ \\ 
$17$&$4E_4\Delta$ \\ 
$19$&$7E_6\Delta$ \\ 
$23$&$4E_4E_6\Delta$ \\ 
$29$&$3E_4^3\Delta+16E_4\Delta^2$ \\ 
$31$&$27E_6^3\Delta+26E_6\Delta^2$ \\ 
$41$&$5E_4E_6^4\Delta+33E_4E_6^2\Delta^2+20E_4\Delta^3$ \\ 
$47$&$6E_4E_6^5\Delta+10E_4E_6^3\Delta^2+16E_4E_6\Delta^3$ \\ 
$59$&$33E_4E_6^7\Delta+4E_4E_6^5\Delta^2+14E_4E_6^3\Delta^3+38E_4E_6\Delta^4$ \\ 
$71$&$4E_4E_6^9\Delta+12E_4E_6^7\Delta^2+65E_4E_6^5\Delta^3+11E_4E_6^3\Delta^4+35E_4E_6\Delta^5$\\ \hline
\end{tabular}
\end{table}

\clearpage

\providecommand{\bysame}{\leavevmode\hbox to3em{\hrulefill}\thinspace}
\providecommand{\MR}{\relax\ifhmode\unskip\space\fi MR }

\providecommand{\MRhref}[2]{  
\href{http://www.ams.org/mathscinet-getitem?mr=#1}{#2}
}
\providecommand{\href}[2]{#2}

\end{document}